\documentclass[12pt]{article}
\usepackage[utf8]{inputenc}

\usepackage{amsmath,amssymb}
\usepackage{amsthm}
\usepackage{multicol}
\usepackage{hyperref}
\usepackage{url}
\usepackage{graphicx}

\usepackage{color}

\usepackage{tikz}
\tikzset{
every node/.style={circle, inner sep=2pt}
}
\usetikzlibrary{matrix,arrows,calc}

\newtheorem{theorem}{Theorem}
\newtheorem{lemma}[theorem]{Lemma}
\newtheorem{proposition}[theorem]{Proposition}
\newtheorem{corollary}[theorem]{Corollary}
\theoremstyle{definition}
\newtheorem{definition}[theorem]{Definition}
\newtheorem{observation}[theorem]{Observation}
\newtheorem{remark}[theorem]{Remark}
\newtheorem{example}[theorem]{Example}
\newtheorem{problem}{Problem}
\newtheorem{question}[theorem]{Question}

\def \I {\mathcal{I}}
\def \S {\mathcal{S}}

\def\Q{\ns Q}

\newcommand{\mytilde}{\raise.17ex\hbox{$\scriptstyle\mathtt{\sim}$}}

\def\j{\mbox{\boldmath $j$}}

\def\vecv{\mbox{\boldmath $v$}}

\def\vec0{\mbox{\boldmath $0$}}

\def\A{\mbox{\boldmath $A$}}
\def\B{\mbox{\boldmath $B$}}

\def\I{\mbox{\boldmath $I$}}

\def\S{\mbox{\boldmath $S$}}

\def\I{\mbox{\boldmath $I$}}

\def\Q{\mbox{\boldmath $Q$}}

\title{Note on the product of the largest \\and the smallest eigenvalue of a graph}
\author{A. Abiad\footnote{Department of Mathematics and Computer Science, Eindhoven University of Technology, The Netherlands, {\tt{a.abiad.monge@tue.nl }}} \and C. Dalf\'o\footnote{Departament  de Matem\`atica, Universitat de Lleida, Igualada (Barcelona), Catalonia, {\tt{cristina.dalfo@udl.cat}}} \and
M. A. Fiol\footnote{Departament de Matem\`atiques, Universitat Polit\`ecnica de Catalunya, Barcelona Graduate School of Mathematics, Institut de Matem\`atiques de la UPC-BarcelonaTech (IMTech),  Barcelona, Catalonia, {\tt{miguel.angel.fiol@upc.edu}}}
}

\date{}

\begin{document}

\maketitle

\begin{abstract}
In this note, we use eigenvalue interlacing to derive an inequality between the maximum degree of a graph and its maximum and minimum adjacency eigenvalues. The case of equality is fully characterized.
\end{abstract}

%%%%%%%%%%%%%%%%%%%%%%%%%%%%%%%%%%%%%%%%%%%%%%%%%%%%%%%%%%%%
\section{Introduction}
%%%%%%%%%%%%%%%%%%%%%%%%%%%%%%%%%%%%%%%%%%%%%%%%%%%%%%%%%%%%
Spectral graph theory seeks to deduce the structural properties of a graph from its spectrum. Eigenvalue interlacing provides a powerful tool for obtaining inequalities and regularity results concerning the structure of graphs in terms of eigenvalues of the adjacency matrix; see, for instance, Brouwer and Haemers \cite{bh11} and Haemers \cite{h95}. In addition, studying the cases of equality in such inequalities often provides interesting information on the structure of the graph.  

Let $G$ be a (connected) graph with $n$ vertices and adjacency matrix having eigenvalues $\lambda_1\geq \cdots \geq \lambda_n$. 
Not many bounds involving the product of the largest and smallest eigenvalue, $\lambda_1$ and $\lambda_n$, of a graph seem to exist. 
In fact, we are only aware of an inequality obtained by Haemers \cite[Theorem 3.3]{h95}, who provides an upper bound for the independence number $\alpha$ of a non-regular graph, extending the celebrated Hoffman ratio bound for regular graphs $\alpha\le n/[1-(\lambda_1/\lambda_n)]$ (see \cite[Theorem 3.2]{h95}). Namely, Haemers proved that, if $G$ is a graph with $n$ vertices, minimum degree $\delta$, and independence number $\alpha$, then
\begin{equation}
\label{haemers}
-\lambda_1\lambda_n\ge \frac{\alpha\delta^2}{n-\alpha}.
\end{equation}

Several papers, such as the one by Gregory, Hershkowitz, and Kirkland \cite{ghk01}, are on the graph spread, that is, $\lambda_1-\lambda_n$. There is also quite some work on bounding the sum $\lambda_1+\lambda_n$, see, for instance, Brandt \cite{B1998}, Csikv\'ari \cite{C2022} and Balogh, Clemen, Lidick\'y, Norin, and Volec \cite{BCLNV2023}. Moreover, bounds on $\lambda_1$ knowing $\lambda_n$, or the other way around, have been obtained by Rojo, Rojo, and Soto \cite{RRS1999}.

The structure of this note is as follows. We start Section \ref{sec:preliminaries} by providing some basic definitions and some known results, such as the interlacing theorem. In Section \ref{sec:inequality}, we present our main result, that is, an inequality involving only the maximum degree $\Delta$ of a graph and the product of its maximum and minimum adjacency eigenvalues. Moreover, the case of equality is fully characterized, and some infinite families satisfying it are provided. Finally, in Section \ref{sec:comparisson} we show that our bound can outperform the other known bound of a similar nature by Haemers \cite[Theorem 3.3]{h95}.

%%%%%%%%%%%%%%%%%%%%%%%%%%%%%%%%%%%%%%%%%%%%%%%%%%%%%%%%%%%%
\section{Preliminaries}\label{sec:preliminaries}
%%%%%%%%%%%%%%%%%%%%%%%%%%%%%%%%%%%%%%%%%%%%%%%%%%%%%%%%%%%%

Throughout this note, $G=(V,E)$ denotes a (simple and connected) graph with $n=|V|$ vertices, and adjacency matrix $\A$ with 
% spectrum
% $$
% \spec G=\{\lambda_0,\lambda_1^{m_1},\cdots,\lambda_d^{m_d}\}.
% $$
% where $\lambda_0>\lambda_1>\cdots>\lambda_d$.
eigenvalues 
$$
\lambda_1(G)\ge \lambda_2(G)\ge \cdots \ge \lambda_n(G).
$$
Moreover, $G(u)$ stands for the set of vertices adjacent to a given vertex $u\in V$.

%\subsection{Eigenvalue interlacing}
Given two square matrices $\A$ and $\B$ with eigenvalues $\lambda_1\geq \cdots \geq \lambda_n$
and 
$\mu_1 \geq \cdots \geq \mu_m$ (with $m<n$), we say that the second sequence {\em interlaces} the first if, for all $i = 1,\ldots,m$, it follows that
$\lambda_i \geq \mu_i \geq \lambda_{n-m+i}$.

A basic result about interlacing is the following (see Haemers \cite{h95}).
\begin{theorem}
	\label{theo-interlacing}
	Let $\S$ be a real $n \times m$ matrix such that $\S^T \S = \I$, and let $\A$ be an $n \times n$ matrix with eigenvalues $\lambda_1 \geq \cdots \geq \lambda_n$. Define $\B = \S^T \A \S$, and call its eigenvalues $\mu_1 \geq\cdots \geq \mu_m$. 
	\begin{enumerate}
		\item[$(i)$]
		The eigenvalues of $\B$ interlace those of $\A$.
		\item[$(ii)$]
		If $\mu_i = \lambda_i$ or $\mu_i = \lambda_{n-m+i}$, then there is an eigenvector $\vecv$ of $\B$ for $\mu_i$ such that $\S \vecv$ is an eigenvector of $\A$ for $\mu_i$.
		\item[$(iii)$]
		If there is an integer $k \in \{0,\ldots,m\}$ such that $\lambda_i = \mu_i$ for $1 \leq i \leq k$,  and $\mu_i = \lambda_{n-m+i}$ for $ k+1 \leq i \leq m$ $(${\em tight interlacing}$)$,  then $\S \B = \A \S$.
	\end{enumerate}
\end{theorem}

Let $\A$ be the adjacency matrix of a graph $G=(V,E)$. We distinguish two interesting cases of eigenvalue interlacing  depending on the matrix $\B$: 

\begin{itemize}
\item[{\bf 1.}]
 	If $\B$ is a {\em principal submatrix} of $\A$, then $\B$ corresponds to the adjacency matrix of an induced subgraph
 	$G'$ of $G$.
\item[{\bf 2.}]
 	If for a given partition of the vertices of $G$, say $V=U_1\cup\cdots\cup U_m$, $\B$ is the so-called {\em quotient matrix} of $\A$. The entries of $\B$, which are denoted by
 	$b_{ij}$ for $i,j=1,\ldots,m$, are the average row sums of the corresponding blocks $\A_{ij}$ of $\A$.
 	
 Moreover, if the interlacing is tight,
 	Theorem~\ref{theo-interlacing}$(iii)$ reflects that $\S$ corresponds to a {\em
 		regular} (or {\em equitable}) partition of $\A$, that is, each
 	block of the partition has constant row and column sums. Then,
 	the bipartite induced subgraphs $G_{ij}$, with adjacency matrices $\A_{ij}$, for $i\neq j$, are biregular, and the subgraphs
 	$G_{ii}$ are regular.
 \end{itemize}

 The {\em cone graph} over a graph $H$ is obtained by adding a new vertex $v$ and joining it to all vertices of $H$. The following result is well known; see, for instance, Brouwer and Haemers \cite[ Ch. 1, Ex. 6]{bh11}.
 \begin{lemma}
 \label{lem1}
 Let $G=H'+v$ be the cone graph on $n$ vertices, over the $\delta$-regular (connected) graph $H'$ with $n-1$ vertices and eigenvalues $\theta_1(=\delta)>\theta_2\ge\cdots\ge \theta_{n-1}$. Then, $G$ has eigenvalues 
 $$
 \theta_2\ge\cdots\ge \theta_{n-1},\quad \mbox{together with}\quad \lambda_{1,2}=\frac{1}{2}\left(d\pm\sqrt{d^2-4(n-1}\right).
 $$
 \end{lemma}
 Notice that $\lambda_1(>\delta)$ is the spectral radius of $G$, whereas, in principle,
 $\lambda_2$ is not necessarily the minimum eigenvalue of $G$.

%%%%%%%%%%%%%%%%%%%%%%%%%%%%%%%%%%%%%%%%%%%%%%%%%%%%%%%%%%%%

\section{The new bound 
% for the product of the largest and smallest eigenvalues
}
\label{sec:inequality}
%%%%%%%%%%%%%%%%%%%%%%%%%%%%%%%%%%%%%%%%%%%%%%%%%%%%%%%%%%%%
In this section we present our main result: an upper bound on the product of the maximum and minimum eigenvalues of a graph. This bound is very useful when we only know information about the maximum degree of the graph.

\begin{theorem}
\label{theo:ineq}
Let $G$ be a graph on $n$ vertices with maximum degree $\Delta=\Delta(G)$ and eigenvalues 
$\lambda_{\max}(G)=\lambda_1(G)\ge\cdots \ge\lambda_n(G)=\lambda_{\min}(G)$. Then,
\begin{equation}
\label{ineq}
-\lambda_{\min}(G){\lambda_{\max}(G)}\ge \Delta.
\end{equation}
 Moreover, equality holds if and only if $G=H'+u$ is a cone graph over a $\delta$-regular graph $H'$ on $\Delta=n-1$ vertices (the degree of $u$) and 
minimum eigenvalue satisfying 
\begin{equation}
\label{cond-mu_n}
\lambda_{\min}(H')\ge \frac{1}{2}\left(\delta - \sqrt{\delta^2 + 4\Delta}\right).
\end{equation}
\end{theorem}

\begin{proof}
Let $u$ be a vertex of $G$ with maximum degree $\Delta$, and consider the graph $H$ induced by the vertex set $\{u\}\cup G(u)$. Let us first prove that $H$ satisfies \eqref{ineq}. If $\A'$ is the adjacency matrix of the graph $H'=\langle G(u)\rangle$ induced by $G(u)$, the adjacency matrix $\A$ of $H$ is of the form
\begin{equation}
\label{partitionA}
\A=\left(
\begin{array}{cc}
0& \j \\
\j^{\top} & \A'\\
\end{array}
\right),
\end{equation}
where $\j$ is the all-1 vector with $\Delta$ entries. Then, its quotient matrix (where each block is replaced by the average sum of its arrows) is
$$
\Q'=\left(
\begin{array}{cc}
0& \Delta \\
1 & \overline{\delta}\\
\end{array}
\right),
$$
where $\overline{\delta}$ is the mean degree of $H'$. 
Since the eigenvalues of $\Q'$ are
$$
\lambda_{1,2}=\frac{1}{2}\left(\overline{\delta}\pm \sqrt{\overline{\delta}^2+4\Delta}\right)
$$
we have, by interlacing,
\begin{equation}
\label{interlac-1}
%\lambda_1(H)\ge \frac{1}{2}\left(\overline{\delta}+ \sqrt{\overline{\delta}^2+4\Delta}\right)>
%\frac{1}{2}\left(\overline{\delta}- \sqrt{\overline{\delta}^2+4\Delta}\right)\ge \lambda_{n+1}(H).
\lambda_{\max}(H)\ge \lambda_1>
\lambda_2\ge \lambda_{\min}(H).
\end{equation}
Then, noting that 
$$
\frac{-\Delta}{\lambda_1}=\frac{-2\Delta}{\overline{\delta}+ \sqrt{\overline{\delta}^2+4\Delta}}=\frac{1}{2}\left(\overline{\delta}- \sqrt{\overline{\delta}^2+4\Delta}\right)=\lambda_2,
$$
we have
\begin{equation}
\label{ineq-2}
-\frac{\Delta(H)}{\lambda_{\max}(H)}\ge -\frac{\Delta}{\lambda_1}=\lambda_2\ge \lambda_{\min}(H),
\end{equation}
and \eqref{ineq} follows.
In fact, as pointed out by Haemers (personal communication, 2024), this can be proved more directly by noting that
$$
-\lambda_{\max}(H)\lambda_{\min}(H)\ge -\lambda_1\lambda_2=-\det(\Q)=\Delta.
$$
Now, since $H$ is an induced subgraph of $G$, we have, again by interlacing, 
$$
\lambda_{\max}(G)\ge \lambda_{\max}(H)> \lambda_{\min}(H)\ge \lambda_{\min}(G).
$$
Thus, using \eqref{ineq-2},
\begin{equation}
\label{ineq-3}
-\frac{\Delta(G)}{\lambda_{\max}(G)}\ge -\frac{\Delta}{\lambda_{\max}(H)}\ge \lambda_{\min}(H)\ge \lambda_{\min}(G),
\end{equation}
and \eqref{ineq} holds.

% Concerning the case of equality, we first iterate  the above procedure to obtain the following chain of inequalities:
% \begin{align}
% -\frac{\Delta}{\lambda_{\max}(G)} & \ge -\frac{\Delta}{\lambda_{\max}(G\setminus v)}\ge \cdots
%   \ge -\frac{\Delta}{\lambda_{\max}(H)}\ge \lambda_{\min}(H)\ge 
%   \cdots \\
% \cdots   & \ge \lambda_{\min}(G\setminus v)\ge\lambda_{\min}(G)
% \end{align}
% where $G=H_{n-\Delta-1}$, $G\setminus v=H_{n-\Delta-2}$,\ldots, $H=H_0$. 

 Now, the equality in \eqref{ineq} first implies  
equalities in \eqref{ineq-2}, that is $\lambda_{\max}(H)=\lambda_1$ and $\lambda_{\min}(H)=\lambda_2$. Consequently, we have tight interlacing, and \eqref{partitionA} corresponds to a regular partition of $H$. That is, $H=H'+v$ is a cone graph over the regular graph $H'$ with degree $\overline{\delta}=\delta$.
Second, assuming that $H$ is a proper subgraph of $G$, the equalities in \eqref{ineq-3} imply that $\lambda_{\max}(G)=\lambda_{\max}(H)$. However, this is impossible since the spectral radius of a proper subgraph of a graph $G$ is known to be smaller than the spectral radius of $G$ (see, for instance, Cvetkovi\'c,  Rowlinson and Simi\'c \cite[Prop. 3.1.10]{crs10}). Therefore, $G=H=H'+u$ is a cone graph, as claimed. Conversely, if $G=H'+v$ is a cone graph over a $\delta$-regular graph $H'$ with minimum eigenvalue satisfying \eqref{cond-mu_n}, by Lemma \ref{lem1}, we have that $\lambda_{\min}(H)=\lambda_{\min}(G)=\lambda_2=\frac{1}{2}\left(\delta - \sqrt{\delta^2 + 4\Delta}\right)$ and, hence, we get the equality in \eqref{ineq}.
\end{proof}

For example, \eqref{cond-mu_n} is satisfied when $G$ is a cone graph of the following type:
\begin{description}
\item[$(i)$] 
$G=K_n=K_{n-1}+u$, the complete graph on $n$ vertices.
\item[$(ii)$]
$G=S_{n-1}=(n-1)K_1+u$, the star graph with $n-1$ rays.
\item[$(iii)$]
$G=W_{n-1}=C_{n-1}+u$, the wheel graph with $n-1$ spokes with $n\ge 8$.
\item[$(iv)$] 
$G=O(k)+v$, where $O(k)$ is the odd graph of degree $k$. 
\item[$(v)$] $Q_k+u$, where $Q_k$ is the $k$-cube of dimension $k$ with $k\ge 7$. 
\end{description}

\vskip.5cm
Indeed, let $\phi(\delta,\Delta)=\frac{1}{2}\left(\delta - \sqrt{\delta^2 + 4\Delta}\right)$, and let us check that the condition \eqref{cond-mu_n} on $H'$ holds.
(Recall that $\Delta$ is used both for the degree of $u$, or maximum degree of $G$,
and the number of vertices of $H'$.)
\begin{description}
\item[$(i)$] The complete graph $H'=K_{n-1}$ on $\Delta=n-1$ vertices, has degree $\delta=n-2$ and its smallest eigenvalue is $\lambda_{\min}(H')=-1=\phi(n-2,n-1)$.

\item[$(ii)$] The graph $H'=(n-1)K_1$ with $\Delta=n-1$ isolated vertices  has degree $\delta=0$,
and its smallest eigenvalue is $\lambda_{\min}(H')=0\ge \phi(0,n-1)=-\sqrt{n-1}$.

\item[$(iii)$] The cycle graph $H'=C_{n-1}$ on $n-1$ vertices has eigenvalues 
 $
 2\cos\left(\frac{2k\pi}{n-1}\right)$ for $k=0,\ldots,n-2$.
 Since $W_3=H'+v\cong K_4$, we can assume that $n\ge 5$. For these values, the minimum of $2\cos\left(\frac{2k\pi}{n-1}\right)$ is when $k=(n/2)-1,n/2$ for even $n$, and when $k=(n-1)/2$ for odd $n$. In the fist case the minimum eigenvalue $\lambda_{\min}(H')=2\cos\left(\frac{(n-2)\pi}{n-1}\right)\ge \phi(2,n-1)=1-\sqrt{n}$, excepting the case for $n=6$. In the second case,  $\lambda_{\min}(H')=-2\ge 1-\sqrt{n}$ excepting when $n=5,7$. 

\item[$(iv)$] The Odd graph $H'=O(k)$ on $\Delta={{2k-1}\choose{k-1}}$ vertices, has degree $\delta=k$ and its smallest eigenvalue is $\lambda_{\min}(H')=-(k-1)\ge \phi(k,\Delta)$. Thus, the successive values of $\phi(k,\Delta)$ for $k=3,4,5,6,\ldots$ are (approximately) $-2$, $-4.245$, $-9$, $-18.702,\ldots$
Then, for instance, $G=P+v$, with $P=O(3)$ being the Petersen graph, has maximum degree $\Delta=10$ and different eigenvalues $5,1,-2$, so satisfying equality in \eqref{ineq}.

\item[$(v)$] The $k$-cube $H'=Q_k$ has $\Delta=2^k$ vertices,  degree $\delta=k$, and its smallest eigenvalue is $\lambda_{\min}(H')=-k\ge \phi(k,2^k)$ when $k\ge 7$.
\end{description}
 
From Theorem \ref{theo:ineq}, we obtain the following straightforward consequences.

\begin{corollary}\leavevmode
\begin{itemize}
\item[$(i)$]
If $G$ is a (non-trivial) regular graph, then its minimum eigenvalue satisfies $\lambda_{\min}\le -1$,
with equality if and only if $G$ is a complete graph.

\item[$(ii)$]
If $G$ is a bipartite graph with maximum degree $\Delta$ and spectral radius $\lambda_{\max}$, then
$\Delta\le \lambda_{\max}^2$, and equality holds if and only if $G=S_{\Delta}$.
\end{itemize}
\end{corollary}

%%%%%%%%%%%%%%%%%%%%%%%%%%%%%%%%%%%%%%%%%%%%%%%%%%%%%%%%%%%%
\section{Bounds comparison}\label{sec:comparisson}
%%%%%%%%%%%%%%%%%%%%%%%%%%%%%%%%%%%%%%%%%%%%%%%%%%%%%%%%%%%%

When equality does not hold in \eqref{ineq}, it is interesting to compare our bound to the one by Haemers \eqref{haemers}. Apart from the fact that \eqref{ineq} is more `economical' in the sense that it uses only the maximum degree, in general, \eqref{ineq} gives a better bound for $-\lambda_1\lambda_n$ when $\alpha$ is not very large. Namely, from $\Delta\ge \alpha\delta^2/(n-\alpha)$, we obtain that \eqref{ineq} outperforms \eqref{haemers} when
$$
\alpha\le \frac{n}{1+\frac{\delta^2}{\Delta}}.
$$

\begin{table}[t]
\label{tab:ind1:tableprop}
\begin{center}
\begin{tabular}{|c|c|}
\hline
Order & Proportion \\
\hline
4&  2/4=0.5\\ 
5 & $14/19\approx 0.737$ \\
6 & $79/107\approx0.738$ \\
7 & $692/849 \approx0.815$ \\
8 & $9489/11100 \approx 0.855$ \\
\hline
\end{tabular}
\end{center}
  \caption{Proportion of small irregular graphs for which the new bound \eqref{ineq} outperforms the known bound \eqref{haemers}.}
\end{table}

In contrast, Haemers' bound \eqref{haemers} is better than \eqref{ineq} for the case of bipartite biregular graphs. In fact, we will show that, for such graphs, equality holds in \eqref{haemers}.
A bipartite graph $G=(V,E)$ with $V = V_1\cup V_2$ is {\em biregular} if all the vertices of the stable set $V_1$ have degree $k_1$, whereas all the vertices of $V_2$ have degree $k_2$. 
Then, notice that by counting in two ways the number of edges of $G$, we have
\begin{equation}
\label{k1n1=k2n2}
k_1n_1 = k_2n_2, 
\end{equation}
where $n_i = |V i|$, for $i = 1, 2$. Since $G$ is bipartite and biregular, its maximum and minimum eigenvalues are $\lambda_1=-\lambda_n=\sqrt{k_1k_2}$ (see, for example, \cite{adf2013} by the authors). Moreover,
assuming that $n_1\ge n_2$, its minimum degree is $\delta=k_1$, and its independence number  is $\alpha(G)=n_1$.
Therefore, using \eqref{k1n1=k2n2}, the bound in \eqref{haemers} become
$$
\frac{\alpha\delta^2}{n-\alpha}=\frac{n_1k_1^2}{n-n_1}=\frac{n_1k_1}{n_2}k_1=k_2k_1=-\lambda_1\lambda_n,
$$
as claimed. Thus, since the maximum degree is $\Delta=k_2$,
Haemers' bound \eqref{haemers} outperforms \eqref{ineq} except when $k_1=1$, which corresponds to the mentioned star graph $S_{k_2}$,
as a special case in Theorem \ref{theo:ineq}.

In fact, this can be seen as a particular case of the following result
by the authors \cite[Proposition 3.1]{adf2013}, which can be seen as the bipartite version of Haemers' bound. 
\begin{proposition}[\cite{adf2013}]
Let $G=(V_1\cup V_2,E)$ be a bipartite
graph with maximum and minimum eigenvalues $\lambda_1$ and $\lambda_n=-\lambda_1$.  Let $\overline{\delta}_i$ be the mean degree of the vertices in $V_i$, for $i=1,2$. Then,
\begin{equation}
\label{haemers-bip}
-\lambda_1\lambda_n=\lambda_1^2\ge \overline{\delta}_1\overline{\delta}_2,    
\end{equation}
with equality if, and only if, $G$ is biregular with degrees $k_1=\overline{\delta_1}$ and $k_2=\overline{\delta}_2$.
\end{proposition}
Then, if $G$ has minimum degree $\delta$ and maximum degree $\Delta$, \eqref{haemers-bip} (or \eqref{haemers} with $\alpha=|V_1|\ge |V_2|$) and \eqref{ineq} imply that, for a bipartite graph, 
$$
-\lambda_1\lambda_n=\lambda_1^2\ge\max\{\delta^2,\Delta\}.
$$
%%%%%%%%%%%%%%%%%%%%%%%%%%%%%%%%%%%%%%%%%%%%%%%%%%%%%%%%%%%%
\subsection*{Acknowledgments}
%%%%%%%%%%%%%%%%%%%%%%%%%%%%%%%%%%%%%%%%%%%%%%%%%%%%%%%%%%%%
The authors thank the useful comments made by Willem Haemers, which helped to improve this paper.
A. Abiad is supported by the Dutch Research Council (NWO) through the grant VI.Vidi.213.085. 
The research of C. Dalf\'o and M. A. Fiol is supported by
AGAUR from the Catalan Government via the project 2021SGR00434 and MICINN from the Spanish Government via the project PID2020-115442RB-I00.
The research of M. A. Fiol is also supported by the  Universitat Polit\`ecnica de Catalunya via the grants AGRUPS-2022 and AGRUPS-2023.

\end{document}